\documentclass[leqno, 12pt]{article}
\usepackage{amsmath,amsfonts,amsthm,amssymb,indentfirst}

\newtheorem{theorem}{Theorem}
\newtheorem{lemma}[theorem]{Lemma}
\newtheorem{definition}[theorem]{Definition}
\newtheorem{corollary}[theorem]{Corollary}
\newtheorem{proposition}[theorem]{Proposition}

\newtheorem{question}[theorem]{Question}
\newtheorem{conjecture}[theorem]{Conjecture}

\theoremstyle{definition}
\newtheorem{example}[theorem]{Example}

\newcommand{\M}{\mathbb{M}}

\begin{document}

\title{Polynomials of small degree evaluated on matrices}
\author{Zachary Mesyan}

\maketitle

\begin{abstract}

A celebrated theorem of Shoda states that over any field $K$ (of characteristic $0$), every matrix with trace $0$ can be expressed as a commutator $AB-BA$, or, equivalently, that the set of values of the polynomial $f(x, y)=xy-yx$ on $\M_n(K)$ contains all matrices with trace $0$. We generalize Shoda's theorem by showing that every nonzero multilinear polynomial of degree at most $3$, with coefficients in $K$, has this property. We further conjecture that this holds for every nonzero multilinear polynomial with coefficients in $K$ of degree $m$, provided that $m-1 \leq n$.

\medskip



\end{abstract}

\section{Introduction}

We begin by recalling a theorem, which was originally proved for fields of characteristic $0$ by Shoda~\cite{Shoda} and later extended to all fields by Albert and Muckenhoupt~\cite{AM}.

\begin{theorem}[Shoda/Albert/Muckenhoupt]\label{Shoda}
Let $K$ be any field, $n\geq 2$ an integer, and $M \in \M_n(K)$. If $M$ has trace $\, 0$, then $M = AB-BA$ for some $A, B \in \M_n(K)$.
\end{theorem} 

For any $A, B \in \M_n(K)$ the traces of $AB$ and $BA$ are equal, and therefore any matrix that can be expressed as a commutator ($AB-BA$) must have trace $0$. (This observation is perhaps what prompted Shoda to prove the above result.) Thus, using $[\M_n(K), \M_n(K)]$ to denote the $K$-subspace of $\M_n(K)$ consisting of the matrices of trace $0$, and $f(\M_n(K))$ to denote the set of values of a polynomial $f$ on $\M_n(K)$, the above theorem can be rephrased as follows.

\begin{corollary} \label{ShodaCor}
Let $K$ be any field, $n\geq 2$ an integer, and $f(x, y) = xy-yx$. Then $f(\M_n(K)) = [\M_n(K), \M_n(K)]$. 
\end{corollary}

It is therefore natural to ask what other polynomials $f$ have the property $f(\M_n(K)) = [\M_n(K), \M_n(K)]$, or more generally, $f(\M_n(K)) \supseteq [\M_n(K), \M_n(K)]$. In this note we conjecture that this is the case for all nonzero multilinear polynomials $f$ (i.e., polynomials that are linear in each variable), provided that the degree of $f$ is at most $n+1$. We then prove this conjecture for multilinear polynomials of degree at most $3$, over fields with at least $n$ elements, and characterize the polynomials that satisfy $f(\M_n(K)) = [\M_n(K), \M_n(K)]$ in this situation. We also prove that for any unital ring $R$, the set of values of the polynomial $xzy-xyz+yzx-zyx$ on $\M_n(R)$ contains all matrices with the property that the sums of the elements along the diagonals above (and including) the main diagonal are $0$.

Since $[\M_n(K), \M_n(K)]$ is a $K$-subspace of $\M_n(K)$, our conjecture and result above can be viewed as a partial answer to a more specific version of the following question of Lvov~\cite[Entry 1.98]{DN}, which also has been attributed to Kaplansky (see~\cite{KMR}).

\begin{question}[Lvov]\label{mainq}
Let $f$ be a multilinear polynomial over a field $K$. Is the set of values of $f$ on the matrix algebra $\, \M_n(K)$ a vector space?
\end{question}

Kanel-Belov, Malev, and Rowen~\cite{KMR} have answered this question in the case where $n=2$ and $K$ is quadratically closed (that is, every non-constant quadratic polynomial over $K$ has a root in $K$). More specifically, they showed that in this case, the image of a multilinear polynomial must be one of $0$, $K\cdot I_2$ (where $I_2$ is the identity matrix), $[\M_2(K), \M_2(K)]$, or $\M_2(K)$. 

Other special cases of Question~\ref{mainq} have been explored elsewhere. For instance, Khurana and Lam~\cite[Corollary 3.16]{DL} showed that the set of values of the multilinear polynomial $xyz-zyx$ on $\M_n(K)$, over an arbitrary field $K$, is all of $\M_n(K)$, when $n\geq 2$ (along with more general versions of this result for other rings $R$ in place of the field $K$). Also, Kaplansky~\cite[Problem 16]{Kap} asked whether there exists a nonzero multilinear polynomial over a field $K$ which always takes values in the center of $\M_n(K)$, or, equivalently, whether there is such a polynomial whose set of values on $\M_n(K)$ is the center of $\M_n(K)$. (This question arose in the study of polynomial identities on rings.) Such polynomials have indeed been constructed by Formanek~\cite{Formanek} and Razmyslov~\cite{Razmyslov}.

All the results and questions mentioned above are stated for multilinear (rather than arbitrary) polynomials, since in general, the set of values of a polynomial on $\M_n(K)$ is not a subspace. More specifically, Chuang~\cite{Chuang} showed that if $K$ is finite, then every subset $S$ of $\M_n(K)$ that contains $0$ and is closed under conjugation is the image of some polynomial. Of course, such subsets $S$ of $\M_n(K)$ are generally not $K$-subspaces. Let us also give an example of a (non-multilinear) polynomial $f$ and an infinite field $K$ such that $f(\M_n(K))$ is not a $K$-subspace of $\M_n(K)$.

\begin{example}
Let $K$ be an algebraically closed field, $n \geq 2$, and $f(x) = x^n$. For all distinct $1\leq i,j \leq n$ we have $f(E_{ii}) = E_{ii}$ and $f(E_{ii}+E_{ij})=E_{ii}+E_{ij}$, where $E_{ij}$ are the matrix units. Hence the $K$-subspace generated by $f(\M_n(K))$ contains all $E_{ij}$, and therefore must be $\M_n(K)$. Thus, to conclude that $f(\M_n(K))$ is not a subspace, it suffices to show that $f(\M_n(K)) \neq \M_n(K)$. Let $A \in \M_n(K)$ be a nonzero nilpotent matrix. Since $0$ is the only eigenvalue of $A$, the Jordan canonical form of $A$ must be strictly upper-triangular, from which it follows that $A^n=0$. Thus, if $A=f(B)=B^n$ for some $B \in \M_n(K)$, then $B$ must be nilpotent, and hence $f(B)=B^n=0$ (by the above argument), contradicting $A \neq 0$. Therefore, $f(\M_n(K))$ contains no nonzero nilpotent matrices, and hence $f(\M_n(K))$ is not a subspace of $\M_n(K)$.  \hfill $\Box$
\end{example}

We conclude the Introduction by mentioning that there has been interest in generalizing Theorem~\ref{Shoda} in other directions, particularly, in determining whether the same statement holds over an arbitrary (unital) ring. That is, given a ring $R$, can every matrix $M \in \M_n(R)$ having trace $0$ be expressed as a commutator $M = AB-BA$, for some $A, B \in \M_n(R)$? Intriguingly, this question remained open until 2000, when Rosset and Rosset~\cite{R&R} produced a ring $R$ and a matrix $M \in \M_2(R)$ having trace $0$ that cannot be expressed as a commutator. On the other hand, it turns out that over any unital ring $R$ and for any $n \geq 2$, every matrix $M \in \M_n(R)$ having trace $0$ can be expressed as a sum of two commutators (see~\cite[Theorem 15]{ZM1}). To the best of our knowledge, however, there is still no classification of the rings $R$ with the property that every matrix with trace $0$ is a commutator in $\M_n(R)$, despite continued attention to the question.

\section{Preliminaries}

Let us next define multilinear polynomials more rigorously, and then collect some basic facts about them.

\begin{definition} \label{multilin-def}
Given a field $K$ and a positive integer $m$, we denote by $K \langle x_1, \dots, x_m \rangle$ the $K$-algebra freely generated by the $($non-commuting$)$ variables $x_1, \dots, x_m$. A polynomial $f(x_1, \dots, x_m) \in K \langle x_1, \dots, x_m \rangle$ is said to be \emph{multilinear} $($\emph{of degree} $m$$)$ if it is of the form $$f(x_1, \dots, x_m) = \sum_{\sigma \in S_m} a_\sigma x_{\sigma(1)} x_{\sigma(2)} \dots x_{\sigma(m)},$$ where $S_m$ is the group of all permutations of $\, \{1,\dots, m\}$ and $a_\sigma \in K$. 
\end{definition}

Some authors call these polynomials \emph{homogeneous multilinear}. The motivation for calling such polynomials ``multilinear" is that given any $K$-algebra $A$, any element $f(x_1, \dots, x_m) \in K \langle x_1, \dots, x_m \rangle$ can be viewed as a map $f: A \times \dots \times A \rightarrow A$ by evaluating the $x_i$ on elements of $A$, and it is easy to see that the polynomials of the above form are precisely the ones that give rise to maps $f: A \times \dots \times A \rightarrow A$ that are linear in each variable (for every $K$-algebra $A$). That is, maps $f$ such that for all $a, b \in K$, $j \in \{1, \dots, m\}$, and $r_1, \dots, r_m, r_j' \in A$, we have $$f(r_1, \dots, r_{j-1}, ar_j+br_j', r_{j+1}, \dots, r_m)$$ $$= af(r_1, \dots, r_{j-1}, r_j, r_{j+1}, \dots, r_m) + bf(r_1, \dots, r_{j-1}, r_j', r_{j+1}, \dots, r_m).$$

Given a polynomial $f(x_1, \dots, x_m) \in K \langle x_1, \dots, x_m \rangle$ and $K$-algebra $A$, we set $f(A) = \{f(r_1, \dots, r_m) \mid r_1, \dots, r_m \in A\}$. Also, for any associative ring $R$ we shall denote by $[r,p]$ the commutator $rp-pr$ of $r,p \in R$, denote by $[R,R]$ the additive subgroup of $R$ generated by the commutators, and let $[r,R] = \{[r,p] \mid p \in R\}$ for $r \in R$.

In the next lemma we record a couple basic facts about multilinear polynomials for later reference. Both claims follow immediately from Definition~\ref{multilin-def}.

\begin{lemma} \label{MultilinPoly}
Let $K$ be a field, $A$ a $K$-algebra, $m$ a positive integer, and $f(x_1, \dots, x_m) \in K \langle x_1, \dots, x_m \rangle$ a multilinear polynomial. Also, let $r_1, \dots, r_m \in A$ be arbitrary elements. Then the following hold.
\begin{enumerate}
\item[$(1)$] For any $a \in K$ we have $af(r_1, \dots, r_m) = f(ar_1, \dots, r_m)$.
\item[$(2)$] For any invertible $p \in A$ we have $pf(r_1, \dots, r_m)p^{-1} = f(pr_1p^{-1} , \dots, pr_mp^{-1} ).$
\end{enumerate}
\end{lemma}

The next lemma, which is an easy consequence of Theorem~\ref{Shoda}, describes the sets of values of multilinear polynomials of degree at most $2$ on matrix algebras.

\begin{lemma}\label{m<3}
Let $K$ be a field, $m$ a positive integer, and $f(x_1, \dots, x_m) \in K\langle x_1, \dots, x_m \rangle$ a multilinear polynomial.
\begin{enumerate}
\item[$(1)$] If $m=1$, then $f(\M_n(K)) \in \{0, \M_n(K)\}$.
\item[$(2)$] If $m=2$, then $f(\M_n(K)) \in \{0, [\M_n(K), \M_n(K)], \M_n(K)\}$.
\end{enumerate}
\end{lemma}

\begin{proof}
If $m=1$, then $f$ must be of the form $f(x)=ax$ for some $a \in K$, from which (1) follows.

If $m=2$, then $f$ must be of the form $f(x,y) = axy+byx$ for some $a, b \in K$. If $a + b \neq 0$, then $f(I_n,Y) = (a+b)Y$ for all $Y\in \M_n(K)$, which implies that $f(\M_n(K)) = \M_n(K)$. While, if $a+b = 0$, then $f(x,y) = a(xy-yx)$, which has the same image as the polynomial $xy-yx$, as long as $a \neq 0$, by Lemma~\ref{MultilinPoly}(1). Statement (2) now follows from Corollary~\ref{ShodaCor}.
\end{proof}

The following fact will be useful in subsequent arguments. For a polynomial $f(x_1, \dots, x_m) \in K\langle x_1, \dots, x_m \rangle$ and $1\leq l\leq m$ we view $f(x_1, \dots, x_l, 1, \dots, 1)$ as a polynomial in $K\langle x_1, \dots, x_l \rangle$.

\begin{corollary} \label{noncomm}
Let $K$ be a field, $m$ a positive integer, and $$f(x_1, \dots, x_m) = \sum_{\sigma \in S_m} a_\sigma x_{\sigma(1)} x_{\sigma(2)} \dots x_{\sigma(m)} \in K\langle x_1, \dots, x_m \rangle$$ a multilinear polynomial.
\begin{enumerate}
\item[$(1)$] If $\, \sum_{\sigma \in S_m} a_\sigma \neq 0,$ then $f(\M_n(K)) = \M_n(K)$. 
\item[$(2)$] If $m \geq 2$ and $f(x_1, x_2, 1, \dots, 1) \neq 0$, then $f(\M_n(K)) \in \{[\M_n(K), \M_n(K)], \M_n(K)\}$.
\end{enumerate}
\end{corollary}

\begin{proof}
This follows from Lemma~\ref{m<3} upon noting that $f(x_1, 1, \dots, 1) = (\sum_{\sigma \in S_m} a_\sigma)x_1$, and $f(x_1, x_2, 1, \dots, 1) = bx_1x_2 + cx_2x_1$, for some $b, c \in K$.
\end{proof}

\section{A conjecture}

The main goal of this section is to justify a conjecture regarding the set of values of certain multilinear polynomials. We shall first require the following result of Amitsur and Rowen~\cite[Proposition 1.8]{AR} to prove a fact about the linear spans of the images of our polynomials.

\begin{proposition}[Amitsur/Rowen] \label{Amitsur}
Let $D$ be a division ring, $n\geq 2$ an integer, and $A \in \M_n(D)$. Then $A$ is similar to a matrix in $\, \M_n(D)$ with at most one nonzero entry on the main diagonal. In particular, if $A$ has trace zero, then it is similar to a matrix in $\, \M_n(D)$ with only zeros on the main diagonal.
\end{proposition}

\begin{proposition} \label{smallm}
Let $K$ be a field, $n\geq 2$ and $m\geq 1$ integers, and $f(x_1, \dots, x_m)$ a nonzero multilinear polynomial in $K \langle x_1, \dots, x_m \rangle$. If $n \geq m-1$, then the $K$-subspace $\langle f(\M_n(K))\rangle$ of $\, \M_n(K)$ generated by $f(\M_n(K))$ contains $\, [\M_n(K), \M_n(K)]$.
\end{proposition}

\begin{proof}
Write $$f(x_1, \dots, x_m) = \sum_{\sigma \in S_m} a_\sigma x_{\sigma(1)} x_{\sigma(2)} \dots x_{\sigma(m)},$$ for some $a_\sigma \in K$. Since $f$ is nonzero, upon relabeling the variables if necessary, we may assume that $a_1 \neq 0$. Furthermore, by Lemma~\ref{m<3}, if $m \leq 2$, then $[\M_n(K), \M_n(K)] \subseteq f(\M_n(K))$ for any $n \geq 2$. We therefore may assume that $m \geq 3$.

Let $i,j \in \{1, \dots, n\}$ be distinct. Since $n \geq m-1 \geq 2$, we can find distinct elements $l_1, \dots, l_{m-3} \in \{1, \dots, n\}\setminus \{i,j\}$. Letting $E_{kl}$ denote the matrix units, we then have $$f(E_{ii}, E_{ij}, E_{jl_1}, E_{l_1l_2}, \dots, E_{l_{m-4}l_{m-3}}, E_{l_{m-3}j})$$ $$= a_1E_{ii} E_{ij} E_{jl_1} E_{l_1l_2} \dots E_{l_{m-4}l_{m-3}} E_{l_{m-3}j} + 0 = a_1E_{ij},$$ since multiplying $E_{ii}, E_{ij}, E_{jl_1}, E_{l_1l_2}, \dots, E_{l_{m-4}l_{m-3}}, E_{l_{m-3}j}$ in any other order yields $0$. (If $m=3$, then we interpret the above equation as $f(E_{ii}, E_{ij}, E_{jj}) = a_1E_{ij}$.) Therefore, $E_{ij} \in f(\M_n(K))$ for all distinct $i$ and $j$, and hence $\langle f(\M_n(K))\rangle$ contains all matrices with zeros on the main diagonal. Now, Lemma~\ref{MultilinPoly} implies that $\langle f(\M_n(K))\rangle$ is closed under conjugation. Hence, by Proposition~\ref{Amitsur}, we have $[\M_n(K), \M_n(K)] \subseteq \langle f(\M_n(K))\rangle$. 
\end{proof}

The claim in the above proposition does not in general hold when $n < m-1$. For example, it is known that the set of values of the polynomial $(xy-yx)^2$ on $\M_2(K)$ is contained in the center $K\cdot I_2$ (see~\cite{Kap}). The same is true of the linearization $$f(x_1, x_2, y_1, y_2) = [x_1,y_1][x_2,y_2] + [x_1,y_2][x_2,y_1] + [x_2,y_1][x_1,y_2] + [x_2,y_2][x_1,y_1]$$ of this polynomial. Thus, $\langle f(\M_2(K)) \rangle = K\cdot I_2 \not\supseteq [\M_2(K), \M_2(K)]$.

An affirmative answer to Question~\ref{mainq} would imply that in the above proposition, if $n \geq m-1$, then $f(\M_n(K))$, and not just $\langle f(\M_n(K)) \rangle$, contains $[\M_n(K), \M_n(K)]$. Since it is suspected that the question does have an affirmative answer (e.g., see~\cite{KMR}), we make the following conjecture.

\begin{conjecture} \label{conjecture}
Let $K$ be a field, $n\geq 2$ and $m\geq 1$ integers, and $f(x_1, \dots, x_m)$ a nonzero multilinear polynomial in $K \langle x_1, \dots, x_m \rangle$. If $n \geq m-1$, then $f(\M_n(K)) \supseteq [\M_n(K), \M_n(K)]$.
\end{conjecture}

Lemma~\ref{m<3} shows that this conjecture holds for $m<3$. In the next section we shall show that it holds for $m=3$ as well, if $K$ has at least $n$ elements.

\section{The three-variable case}

We shall require another fact proved by Amitsur and Rowen~\cite[Lemma 1.2]{AR}.

\begin{lemma}[Amitsur/Rowen] \label{Amitsur2}
Let $K$ be a field and $n \geq 2$ an integer. Suppose that $A = (a_{ij}) \in \M_n(K)$ is a diagonal matrix with $a_{ii}\neq a_{jj}$ for $i\neq j$. Then $\, [A,\M_n(K)]$ consists of all the matrices with only zeros on the main diagonal.
\end{lemma}

We are now ready for our main result.

\begin{theorem} \label{mainth}
Let $n \geq 2$ be an integer, $K$ a field with at least $n$ elements, and $f \in K \langle x, y, z \rangle$ any nonzero multilinear polynomial. Then $f(\M_n(K))$ contains every matrix having trace $\, 0$.
\end{theorem}

\begin{proof}
If $f$ has degree at most $2$, then this follows from Lemma~\ref{m<3}. Thus, let us assume that the degree of $f$ is $3$. We can then write $$f(x,y,z) = axyz + bxzy + cyxz + dyzx + ezxy + gzyx \ (a, b, c, d, e, g \in K).$$ If $a+b+c+d+e+g \neq 0$, then $[\M_n(K), \M_n(K)] \subseteq f(\M_n(K))$, Corollary~\ref{noncomm}(1). Let us therefore suppose that $a+b+c+d+e+g = 0$. In this case $$f(x,y,z) = a(xyz-zyx) + b(xzy-zyx) + c(yxz-zyx) + d(yzx-zyx) + e(zxy-zyx).$$ Moreover, if any of $f(1,y,z)$, $f(x,1,z)$, or $f(x,y,1)$ are nonzero, then by Corollary~\ref{noncomm}(2), $[\M_n(K), \M_n(K)] \subseteq f(\M_n(K))$. Thus, let us assume that $$0=f(1,y,z)=a(yz-zy) + c(yz-zy) + d(yz-zy) = (a+c+d)(yz-zy),$$ which implies that $0 = a+c+d$. Setting $0=f(x,1,z)$ and $0=f(x,y,1)$ we similarly get $0 = a+b+c$ and $0 = a+b+e$. Solving the resulting system of equations gives $b=d$, $c=e$, and $a=-b-c$. Therefore, $$f(x,y,z) = (-b-c)(xyz-zyx) + b(xzy-zyx+yzx-zyx) + c(yxz-zyx+zxy-zyx)$$ $$=b(xzy-zyx+yzx-xyz)+c(yxz-zyx+zxy-xyz)$$ $$=b(x[z,y]-[z,y]x) + c(z[x,y]-[x,y]z) = b[x,[z,y]] + c[z,[x,y]].$$ Now, since $f(x,y,z) \neq 0$, renaming the variables, if necessary, we may assume that $b \neq 0$. Then, by Lemma~\ref{MultilinPoly}(1), $f(x,y,z)$ and $b^{-1}f(x,y,z)$ have the same set of values. Therefore, we may assume that $f(x,y,z) = [x,[z,y]] + b[z,[x,y]]$ for some $b \in K$. 

Let $A \in \M_n(K)$ be a matrix having trace $0$. We wish to show that $A \in f(\M_n(K))$. By Proposition~\ref{Amitsur}, $A$ is conjugate to a matrix $A' \in \M_n(K)$ with only zeros on the main diagonal. Hence, by Lemma~\ref{MultilinPoly}(2), to conclude the proof it is enough to show that that $A' \in f(\M_n(K))$. By our assumption on $K$, we can find a diagonal matrix $M \in \M_n(K)$ with distinct elements of $K$ on its main diagonal. By Lemma~\ref{Amitsur2}, there is some $B \in \M_n(K)$ such that $A'=[M,B]$. Now, write $B = C+D$, where $C$ has only zeros on the main diagonal and $D$ is diagonal. Then $$A'=[M,B]=[M,C+D]=[M,C]+[M,D]=[M,C],$$ since $M$ commutes with all diagonal matrices. Using Lemma~\ref{Amitsur2} once again, we can find a matrix $E \in \M_n(K)$ such that $C=[E,M]$. Finally, we have $$f(M,M,E) = [M,[E,M]] + b[E,[M,M]] = [M,[E,M]] = [M,C]=A',$$ as desired.
\end{proof}

Let us next describe the degree-three multilinear polynomials $f$ satisfying $f(\M_n(K)) = [\M_n(K), \M_n(K)]$.

\begin{lemma}
Let $n \geq 2$ be an integer, $K$ a field, and $$f(x_1, x_2, x_3) = \sum_{\sigma \in S_3} a_\sigma x_{\sigma(1)} x_{\sigma(2)} x_{\sigma(3)} \in K \langle x_1, x_2, x_3 \rangle$$ a multilinear polynomial of degree $3$. Then $f(\M_n(K)) \subseteq [\M_n(K), \M_n(K)]$ if and only if $\, \sum_{\sigma \in S_3} a_\sigma=0 = \sum_{\sigma \in A_3} a_\sigma,$ where $A_3 \subseteq S_3$ is the alternating subgroup.
\end{lemma}

\begin{proof}
By Corollary~\ref{noncomm}(1), we may assume that $\sum_{\sigma \in S_3} a_\sigma=0$. Now, suppose that $\sum_{\sigma \in A_3} a_\sigma = 0$. Then we must also have $\sum_{\sigma \in S_3\setminus A_3} a_\sigma = \sum_{\sigma \in S_3}  a_\sigma- \sum_{\sigma \in A_3} a_\sigma = 0$. Therefore, $$f(x_1, x_2, x_3) = \sum_{\sigma \in A_3} a_\sigma x_{\sigma(1)} x_{\sigma(2)} x_{\sigma(3)} + \sum_{\sigma \in S_3 \setminus A_3} a_\sigma x_{\sigma(1)} x_{\sigma(2)} x_{\sigma(3)}$$ $$=\sum_{\sigma \in A_3 \setminus \{1\}} a_\sigma (x_{\sigma(1)} x_{\sigma(2)} x_{\sigma(3)} -x_1x_2x_3)+ \sum_{\sigma \in S_3 \setminus (A_3 \cup \{(12)\})} a_\sigma (x_{\sigma(1)} x_{\sigma(2)} x_{\sigma(3)}- x_2x_1x_3)$$ $$= a_{(123)}(x_2 x_3 x_1 -x_1x_2x_3) + a_{(132)}(x_3 x_1 x_2 -x_1x_2x_3)$$  $$+ a_{(13)}(x_3 x_2 x_1 -x_2x_1x_3) + a_{(23)}(x_1 x_3 x_2 -x_2x_1x_3)$$ $$= a_{(123)}[x_2x_3,x_1] + a_{(132)}[x_3,x_1x_2] + a_{(13)}[x_3,x_2x_1] + a_{(23)}[x_1x_3,x_2].$$ Thus, $f(\M_n(K)) \subseteq [\M_n(K), \M_n(K)]$.

On the other hand, $f(E_{11},E_{12},E_{21}) = a_1E_{11} + a_{(123)}E_{11} + a_{(132)}E_{22}$ has trace $\sum_{\sigma \in A_3} a_\sigma$. Hence, if this sum is not zero, then $f(\M_n(K)) \not\subseteq [\M_n(K), \M_n(K)]$.
\end{proof}

Applying Theorem~\ref{mainth} to this lemma we obtain the following.

\begin{corollary}
Let $n \geq 2$ be an integer, $K$ a field with at least $n$ elements, and $$f(x_1, x_2, x_3) = \sum_{\sigma \in S_3} a_\sigma x_{\sigma(1)} x_{\sigma(2)} x_{\sigma(3)} \in K \langle x_1, x_2, x_3 \rangle$$ a nonzero multilinear polynomial of degree $3$. Then $f(\M_n(K)) = [\M_n(K), \M_n(K)]$ if and only if $\, \sum_{\sigma \in S_3} a_\sigma=0 = \sum_{\sigma \in A_3} a_\sigma,$ where $A_3 \subseteq S_3$ is the alternating subgroup.
\end{corollary}

We note that in general the condition $\sum_{\sigma \in S_m} a_\sigma=0 = \sum_{\sigma \in A_m} a_\sigma$ does not characterize the nonzero multilinear polynomials $f(x_1, \dots, x_m)$ satisfying $f(\M_n(K)) = [\M_n(K), \M_n(K)]$. For example, if $m=2$, then $\sum_{\sigma \in S_m} a_\sigma=0 = \sum_{\sigma \in A_m} a_\sigma$ implies that $f(x_1, x_2)=0$. Also, $f(x_1, x_2, x_3, x_4) = x_1x_2x_3x_4-x_4x_3x_2x_1$ satisfies $\sum_{\sigma \in S_4} a_\sigma=0 = \sum_{\sigma \in A_4} a_\sigma$, but $f(\M_n(K)) \not\subseteq [\M_n(K), \M_n(K)]$ for any $n \geq 2$, since $f(E_{11},E_{12},E_{22},E_{21}) = E_{11} \notin [\M_n(K), \M_n(K)]$.

Let us conclude with a fact about the image of the degree-three multilinear polynomial $[x,[y,z]]$ on matrices over an arbitrary ring. We first require the following lemma.

\begin{lemma}\label{crux}
Let $R$ be a unital ring, let $n\geq 2$ be an integer, and let $A = (a_{ij}) \in \M_n(R)$ be such that for each $j \in \{0, 1, \dots, n-1\}$ we have $\, \sum_{i=1}^{n-j} a_{i, i+j} = 0$ $($i.e., sums of the elements along the diagonals above and including the main diagonal are $\, 0$$)$. Then $A = DX-XD$ for some $D \in \M_n(R)$ and $X = \sum_{i=1}^{n-1} E_{i+1,i}$, where $E_{ij}$ are the matrix units.
\end{lemma}

\begin{proof}
Letting $Z = \sum_{i=1}^{n-1} E_{i,i+1}$ we have $ZX = \sum_{i=1}^{n-1}E_{ii} = I - E_{nn}.$ For any $l \in \{0,1, \dots, n-1\}$ and $k \in \{1, 2, \dots, n\}$ we then have $$E_{kk}X^lAZ^lE_{nn} = E_{kk} \bigg( \sum_{i=1}^{n-l} E_{i+l,i} \bigg) A \bigg( \sum_{i=1}^{n-l} E_{i,i+l} \bigg) E_{nn} = E_{k,k-l}AE_{n-l,n} = a_{k-l, n-l}E_{kn}$$ if $l < k$, and $E_{kk}X^lAZ^lE_{nn} = 0 \cdot AZ^lE_{nn} = 0$ if $l \geq k$.

Letting $D = \sum_{i=0}^{n-2} X^iAZ^{i+1}$, we have $$DX - XD = \bigg( \sum_{i=0}^{n-2} X^iAZ^i \bigg) ZX - \sum_{i=0}^{n-2} X^{i+1}AZ^{i+1}$$ $$= \bigg( A + \sum_{i=1}^{n-2} X^iAZ^i \bigg) (I - E_{nn}) - \sum_{i=1}^{n-1} X^iAZ^i = A - \bigg( \sum_{i=0}^{n-2} X^iAZ^i \bigg) E_{nn} - X^{n-1}AZ^{n-1}$$ $$= A - \bigg( \sum_{i=0}^{n-2} X^iAZ^i \bigg)E_{nn} - X^{n-1}AE_{1n} = A - \bigg( \sum_{i=0}^{n-1} X^iAZ^i \bigg) E_{nn}.$$ Now, for every $k \in \{1, 2, \dots, n\}$ we have $$E_{kk} \bigg( \sum_{i=0}^{n-1} X^iAZ^i \bigg) E_{nn} = \sum_{i=0}^{k-1} E_{kk}X^iAZ^iE_{nn} = \bigg( \sum_{i=0}^{k-1} a_{k-i, n-i} \bigg) E_{kn} = \bigg( \sum_{i=1}^{k} a_{i, i+(n-k)} \bigg) E_{kn},$$ by the computation in the first paragraph. Finally, the last sum is $0$, by hypothesis on $A$, and hence $(\sum_{i=0}^{n-1} X^iAZ^i)E_{nn} = 0$, showing that $DX - XD = A$.
\end{proof}

\begin{theorem} \label{zero-sup}
Let $R$ be unital ring, let $n \geq 2$ be an integer, let $f(x,y,z) = [x,[z,y]]$, and let $A = (a_{ij}) \in \M_n(R)$ be such that for each $j \in \{0, 1, \dots, n-1\}$ we have $\, \sum_{i=1}^{n-j} a_{i, i+j} = 0$. Then $A = f(D, X, Y)$ for some $D \in \M_n(R)$, $X = \sum_{i=1}^{n-1} E_{i+1,i}$, and $Y = \sum_{i=1}^{n} (i-1)E_{ii}$.
\end{theorem}

\begin{proof}
For any matrix $M = (m_{ij}) \in \M_n(K)$ we have $$[Y,M] 
=\left(\begin{array}{ccccc}
0 & 0 & 0 & \dots & 0 \\
m_{21} & m_{22} & m_{23} & \dots & m_{2n} \\
2m_{31} & 2m_{32} & 2m_{33} & \dots & 2m_{3n} \\
\vdots & \vdots & \vdots & \ddots & \vdots \\
\end{array}\right) - 
\left(\begin{array}{ccccc}
0 & m_{12} & 2m_{13} & \dots & (n-1)m_{1n} \\
0 & m_{22} & 2m_{23} & \dots & (n-1)m_{2n} \\
0 & m_{32} & 2m_{33} & \dots & (n-1)m_{3n} \\
\vdots & \vdots & \vdots & \ddots & \vdots \\
\end{array}\right)$$
$$= ((i-1)m_{ij})-((j-1)m_{ij}) = ((i-j)m_{ij}).$$ Hence, in particular, $[Y,X] = X$.

Now, by Lemma~\ref{crux}, $A = [D,X]$ for some $D \in \M_n(R)$. We therefore have $$f(D, X, Y) = [D,[Y,X]] = [D,X]=A,$$ proving the statement.
\end{proof}

This argument extends easily to all polynomials of the form $[x_1,[x_2, \dots, [x_{m-1},x_m]]\dots]$.

\section*{Acknowledgements} I am grateful to Mikhail Chebotar for bringing Question~\ref{mainq} to my attention, and to the referee for suggesting a better way to prove the main result.

\vspace{.1in}

\noindent
Department of Mathematics \newline
University of Colorado \newline
Colorado Springs, CO 80918 \newline
USA \newline

\noindent Email: \tt{zmesyan@uccs.edu}


\begin{thebibliography}{00}
\bibitem{AM} A.\ A.\ Albert and B.\ Muckenhoupt,
\textit{On matrices of trace zero,} Michigan Math.\ J.\ \textbf{4}
(1957) 1--3.

\bibitem{AR} S.\ A.\ Amitsur and L.\ H.\ Rowen, \textit{Elements of reduced trace 0,} Israel J.\ Math. \textbf{87} (1994) 161--179.

\bibitem{Chuang} C.-L.\ Chuang, \textit{On ranges of polynomials in finite matrix rings,} Proc.\ Amer.\ Math.\ Soc.\ \textbf{110} (1990) 293--302.

\bibitem{DN} \textit{The Dniester Notebook: Unsolved Problems in the Theory of Rings and Modules,} Mathematics Institute, Russian Academy of Sciences Siberian Branch, Novosibirsk, Fourth Edition, 1993.

\bibitem{Formanek} E.\ Formanek, \textit{Central polynomials for matrix rings,} J.\ Algebra \textbf{23} (1972) 129--132.

\bibitem{KMR} A.\ Kanel-Belov, S.\ Malev, and L.\ Rowen, \textit{The images of non-commutative polynomials evaluated on $2\times 2$ matrices,} Proc.\ Amer.\ Math.\ Soc.\ \textbf{140} (2012) 465--478. 

\bibitem{Kap} I.\ Kaplansky, \textit{``Problems in the Theory of Rings" Revisited,} Amer.\ Math.\ Monthly \textbf{77} (1970) 445--454.

\bibitem{DL} D.\ Khurana and T.\ Y.\ Lam, \textit{Generalized commutators in matrix rings,} Lin.\ Multilin.\ Alg.\ \textbf{60} (2012) 797--827.

\bibitem{ZM1} Z.\ Mesyan, \textit{Commutator rings,} Bull.\ Austral.\ Math Soc.\ \textbf{74} (2006) 279--288.

\bibitem{Razmyslov} Y.\ Razmyslov, \textit{On a problem of Kaplansky,} Izv.\ Akad.\ Nauk SSSR., Ser.\ Mat.\ \textbf{37} (1973) 483--501.

\bibitem{R&R} M.\ Rosset and S.\ Rosset, \textit{Elements of trace zero that are not commutators,} Comm.\ Algebra \textbf{28} (2000) 3059--3072.

\bibitem{Shoda} K.\ Shoda, \textit{Einige S\"{a}tze \"{u}ber Matrizen,} Japan J.\ Math \textbf{13} (1936) 361--365.
\end{thebibliography}
\end{document}